\documentclass[11pt]{article}

\usepackage{geometry}
\usepackage{microtype}

\usepackage{graphicx}
\graphicspath{{./figs/}}

\usepackage{authblk}
\usepackage{amsmath}

\makeatletter
\DeclareMathOperator*{\infd}{inf\vphantom{\operator@font p}}
\DeclareMathOperator*{\supx}{su\smash{\operator@font p}}
\makeatother

\usepackage{amsthm}
\usepackage{amssymb}
\usepackage{multicol}
\usepackage{marginnote}
\usepackage{todonotes}
\usepackage{tikz-cd}

\newcommand{\Cech}{\v{C}ech}
\newcommand{\C}[1]{\mbox{\v{C}ech}{(#1)}}
\newcommand{\R}{\mathbb{R}}

\newcommand{\ds}{\displaystyle}
\newcommand{\ol}{\overline}

\newcommand{\Dg}[1]{\mathrm{Dg}(#1)}

\newcommand{\DgPD}[1]{\Dg {#1}}
\newcommand{\feas}[1]{F_{\ol{#1}}}

\newcommand{\denselist}{\itemsep 0pt\parsep=1pt\partopsep 0pt}

\newcommand{\myheight}		{{height}}

\newtheorem{thm}{Theorem}

\newtheorem{lem}[thm]{Lemma}
\newtheorem{prop}[thm]{Proposition}
\newtheorem{cor}[thm]{Corollary}
\newtheorem{conj}[thm]{Conjecture}
\newtheorem{obs}{Observation}
\newtheorem{defn}[thm]{Definition}

\newcommand {\mm}[1] {\ifmmode{#1}\else{\mbox{\(#1\)}}\fi}

\title{The Relationship Between the Intrinsic \Cech{} and Persistence Distortion Distances for Metric Graphs
\thanks{The authors are grateful for the Women in Computational Topology (WinCompTop) workshop at the Institute for Mathematics and its Applications (IMA) for initiating our research collaboration, and participant travel support to this workshop was made possible through the National Science Foundation grant NSF-DMS-1619908. The IMA also hosted this collaborative group at a follow-up workshop. The American Institute of Mathematics (AIM) supports our group through the Structured Quartet Research Ensembles (SQuaRE) program. EP is supported by Pacific Northwest National Laboratory, operated for the U.S. Department of Energy by Battelle under Contract DE- AC05-76RL01830. RS is partially supported by the Simons Collaboration Grant 318086. YW is partially supported by National Science Foundation under grants RI-1815697, CCF-1740761, and CCF-1618247. }}

\author[1]{Ellen Gasparovic\thanks{gasparoe@union.edu}}
\author[2]{Maria Gommel\thanks{maria-gommel@uiowa.edu}}
\author[3]{Emilie Purvine\thanks{emilie.purvine@pnnl.gov}}
\author[4]{Radmila Sazdanovic\thanks{rsazdanovic@math.ncsu.edu}}
\author[5]{Bei Wang\thanks{beiwang@sci.utah.edu}}
\author[6]{Yusu Wang\thanks{yusu@cse.ohio-state.edu}}
\author[7]{Lori Ziegelmeier\thanks{lziegel1@macalester.edu}}
\affil[1]{Department of Mathematics, Union College, Schenectady, NY}
\affil[2]{Department of Mathematics, University of Iowa, Iowa City, IA}
\affil[3]{Computing \& Analytics Division, Pacific Northwest National Laboratory, Seattle, WA}
\affil[4]{Department of Mathematics, North Carolina State University, Raleigh, NC}
\affil[5]{School of Computing and Scientific Computing and Imaging Institute, University of Utah, Salt Lake City, UT}
\affil[6]{Department of Computer Science, Ohio State University, Columbus, OH}
\affil[7]{Department of Mathematics, Statistics, and Computer Science, Macalester College, Saint Paul, MN}

\date{}

\begin{document}
\maketitle
\begin{abstract}
Metric graphs are meaningful objects for modeling complex structures that arise in many real-world applications, such as road networks, river systems, earthquake faults, blood vessels, and filamentary structures in galaxies.
To study metric graphs in the context of comparison, we are interested in determining the relative discriminative capabilities of two topology-based distances between a pair of arbitrary finite metric graphs: the persistence distortion distance and the intrinsic \Cech{} distance.
We explicitly show how to compute the intrinsic \Cech{} distance between two metric graphs based solely on knowledge of the shortest systems of loops for the graphs.
Our main theorem establishes an inequality between the intrinsic \Cech{} and persistence distortion distances in the case when one of the graphs is a bouquet graph and the other is arbitrary.
The relationship also holds when both graphs are constructed via wedge sums of cycles and edges.
\end{abstract}


\section{Introduction}

When working with graph-like data equipped with a notion of distance, a very useful means of capturing existing geometric and topological relationships within the data is via a \emph{metric graph}.
Given an ordinary graph $G=(V,E)$ and a length function on the edges, one may view $G$ as a metric space with the shortest path metric in any geometric realization.

Metric graphs are used to model a variety of real-world data sets, such as road networks, river systems, earthquake faults, blood vessels, and filamentary structures in galaxies~\cite{AanjaneyaChazalChen2012,Sousbie2011,TupinMaitreMangin1998}.
Given these practical applications, it is natural to ask how to compare two metric graphs in a meaningful way.
Such a comparison is important to understand the stability of these  structures in the noisy setting.
One way to do this is to check whether there is a bijection between the two input graphs as part of a graph isomorphism problem~\cite{Babai2016}.
Another way is to define, compute, and compare various distances on the space of graphs.
In this paper, we are interested in determining the discriminative capabilities of two distances that arise from computational topology: the persistence distortion distance and the intrinsic \Cech{} distance.
If two distances $d_1$ and $d_2$ on the space of metric graphs satisfy an inequality $d_1(G_1, G_2) \leq c \cdot d_2(G_1, G_2)$ (for some constant $c>0$ and any pair of graphs $G_1$ and $G_2$), this means that $d_2$ has greater discriminative capacity for differentiating between two input graphs.
For instance, if $d_1(G_1, G_2) = 0$ and $d_2(G_1, G_2) > 0$, then $d_2$ has a better discriminative power than $d_1$.

\subsection{Related work}
Well-known methods for comparing graphs using distance measures include combinatorial (e.g.,~graph edit distance~\cite{ZengTungWang2009}) and spectral (e.g.,~eigenvalue decomposition~\cite{Umeyama1998}) approaches.
Graph edit distance minimizes the cost of transforming one graph to another via a set of elementary operators such as node/edge insertions/deletions,
     while spectral approaches optimize objective functions based on properties of the graph spectra.

Recently, several distances for comparing metric graphs have been proposed based on ideas from computational topology.
In the case of a special type of metric graph called a Reeb graph, these distances include: the functional distortion distance~\cite{BauerGeWang2014}, the combinatorial edit distance~\cite{FabioLandi2016}, the interleaving distance~\cite{SilvaMunchPatel2016}, and its variant in the setting of merge trees~\cite{MorozovBeketayevWeber2013}.
In particular, the functional distortion distance can be considered as a variation of the Gromov-Hausdorff distance between two metric spaces~\cite{BauerGeWang2014}.
The interleaving distance is defined via algebraic topology and utilizes the equivalence between Reeb graphs and cosheaves~\cite{SilvaMunchPatel2016}.
For metric graphs in general, both the persistence distortion distance~\cite{DeyShiWang2015} and the intrinsic \Cech{} distance~\cite{ChazalSilvaOudot2014} take into consideration the structure of metric graphs, independent of their geometric embeddings, by treating them as continuous metric spaces.
In \cite{OudotSolomon2018}, Oudot and Solomon
point out that since compact geodesic spaces can be approximated by finite metric graphs in the Gromov--Hausdorff sense~\cite{BuragoBuragoIvanov2001} (see also the recent work of M\'emoli and Okutan \cite{MemoliOkutan2018}), one can
study potentially complicated length spaces by studying the persistence distortion of a sequence of approximating graphs.

In the context of comparing the relative discriminative capabilities of these distances, Bauer, Ge, and Wang~\cite{BauerGeWang2014} show that the functional distortion distance between two Reeb graphs is bounded from below by the bottleneck distance between the persistence diagrams of the Reeb graphs.
Bauer, Munch, and Wang~\cite{BauerMunchWang2015} establish a strong equivalence between the functional distortion distance and the interleaving distance on the space of all Reeb graphs, which implies the two distances are within a constant factor of one another. Carri\`ere and Oudot \cite{CarriereOudot2017} consider the intrinsic versions of the aforementioned distances and prove that they are all globally equivalent. They also establish a lower bound for the bottleneck distance in terms of a constant multiple of the functional distortion distance. In \cite{DeyShiWang2015}, Dey, Shi, and Wang show that the persistence distortion distance is stable with respect to changes to input metric graphs as measured by the Gromov-Hausdorff distance. In other words, the persistence distortion distance is bounded above by a constant factor of the Gromov-Hausdorff distance.
Furthermore, the intrinsic \Cech{} distance is also bounded from above by the Gromov-Hausdorff distance for general metric spaces~\cite{ChazalSilvaOudot2014}.

\subsection{Our contribution}

The main focus of this paper is relating two specific topological distances between general metric graphs $G_1$ and $G_2$: the intrinsic \Cech{} distance and the persistence distortion distance.
Both of these can be viewed as distances between topological signatures or summaries of $G_1$ and $G_2$.
Indeed, in the case of the intrinsic \Cech{} distance, a metric graph $(G, d_G)$ is mapped to the persistence diagram $\mathrm{Dg}_1 IC_G$ induced by the so-called intrinsic \Cech{} filtration $IC_G$, and we may think of $\mathrm{Dg}_1 IC_G$ as the signature of $G$.
The intrinsic \Cech{} distance $d_{IC} (G_1, G_2)$ between two metric graphs $G_1$ and $G_2$ is 
the bottleneck distance between these signatures, denoted $d_B(\mathrm{Dg}_1 IC_{G_1}, \mathrm{Dg}_1 IC_{G_2})$.
For the persistence distortion distance, each metric graph $G$ is mapped to a set $\Phi(G)$ of persistence diagrams, which is the signature of the graph $G$ in this case.
The persistence distortion distance $d_{PD}(G_1, G_2)$ between $G_1$ and $G_2$ is measured by the Hausdorff distance between these image sets or signatures. See Section~\ref{sec:bground} for the definition of $\Phi$, along with more detailed definitions of these two distances.

Our objective is to determine the relative discriminative capacities of such signatures.
We conjecture that the persistence distortion distance is 
more discriminative than the intrinsic \Cech{} distance.

\begin{conj}  $d_{IC} \le c \cdot d_{PD}$ for some constant $c > 0$. \label{mainCon}
\end{conj}
It is known from \cite{GasparovicGommelPurvine2018} that $\mathrm{Dg}_1 IC_G$ depends only on the lengths of the shortest system of loops in $G$, and thus the persistence distortion distance appears to be more discriminative, intuitively.
We show in Section~\ref{sec:icdist} that the intrinsic \Cech{} distance between two arbitrary finite metric graphs is determined solely by the difference in these shortest cycle lengths; see Theorem \ref{thm:ICDist} for a precise statement.
This further implies that the intrinsic \Cech{} distance between two arbitrary metric trees is always $0$.
In contrast, the persistence distortion distance takes relative positions of loops as well as branches into account, and is nonzero in the case of two trees.
In other words, the conjecture holds for metric trees.

We make progress toward proving the conjecture in greater generality in this paper.
Theorem \ref{thm:main}
establishes an inequality 
between the intrinsic \Cech{} and persistence distortion distances for two finite metric graphs in the case when one of the graphs is a bouquet graph and the other is arbitrary. In this case, the constant $c=1/2$ so that the inequality is sharper than what is conjectured.
The theorem and proof appear in Section~\ref{sec:main}, and we conclude that section by proving that Conjecture \ref{mainCon} also holds
when both graphs are constructed by taking wedge sums of cycles and edges.
While this does not yet prove the conjecture for arbitrary metric graphs, our work provides the first non-trivial relationship between these two meaningful topological distances. Our proofs also provide insights on the map $\Phi$ from a metric graph into the space of persistence diagrams as utilized in the definition of the persistence distortion distance. This map $\Phi$ is of interest itself; indeed, see the recent study of this map in \cite{OudotSolomon2018}.

In general, we believe that this direction of establishing qualitative understanding of topological signatures and their corresponding distances is interesting and valuable for use in applications.
We leave the proof of the conjecture for arbitrary metric graphs as an open problem and give a brief discussion on some future directions
in Section~\ref{sec:future}.

%

\section{Background}
\label{sec:bground}

\subsection{Persistent homology and metric graphs}
\label{subsec:ph+graphs}

We begin with a brief summary of \emph{persistent homology} and how it can be utilized in the context of \emph{metric graphs}. For background on homology and simplicial complexes, we refer the reader to \cite{Hatcher2002,Munkres1984}, and for further details on persistent homology, see, e.g., \cite{Carlsson2009,EdelsbrunnerHarer2008}.

In persistent homology, one studies the changing homology of an increasing sequence of subspaces of a topological space $X$.
One (typical) way to obtain a filtration of $X$ is to take a continuous function $f:X \rightarrow \mathbb{R}$ and construct the \emph{sublevel set filtration}, $\emptyset = X_{a_0} \subseteq X_{a_1} \subseteq \ldots \subseteq X_{a_m}=X$, by writing $X_{a_i} = f^{-1}(-\infty,a_i)$ for the sublevel set defined by the value $a_i$.
The inclusions $\{X_{a_i} \to X_{a_{j}}\}_{0 \leq i < j \leq m}$ induce the \emph{persistence module} $H_k(X_{a_0}) \rightarrow  H_k(X_{a_1})  \rightarrow \ldots \rightarrow H_k(X_{a_m})$ in any homological dimension $k$ by applying the homology functor with coefficients in some field.
Another way to obtain a filtration is to build a sequence of simplicial complexes on a set of points using, for instance, the \emph{intrinsic \Cech{} filtration} \cite{ChazalSilvaOudot2014} discussed in Section~\ref{sec:distancesbground}.


Elements of each homology group may then be tracked through the filtration and recorded in a \emph{persistence diagram}, with one diagram for each $k$. A persistence diagram is a multiset of points $(a_i, a_j)$ in the extended plane $(\R\cup \pm \infty)^2$, where each point $(a_i, a_j)$ corresponds to a homological element that appears for the first time (is ``born'') at $H_k(X_{a_i})$ and which disappears (``dies'') at $H_k(X_{a_j})$.
A persistence diagram also includes the infinitely many points along the diagonal line $y=x$.
The usual mantra for persistence is that points close to the diagonal are likely to represent noise, while points further from the diagonal may encode more robust topological features.

In this paper, we are interested in summarizing the topological structure of a finite \emph{metric graph}, specifically in homological dimension $k=1$.
Given a graph $G=(V,E)$, where $V$ and $E$ denote the vertex and edge sets, respectively, as well as a length function, $length: E \rightarrow \mathbb{R}_{\geq 0}$, on edges in $E$, a finite metric graph $(|G|,d_G)$ is a metric space where $|G|$ is a geometric realization of $G$ and $d_G$ is defined as in \cite{DeyShiWang2015}. Namely, if $e$ and $|e|$ denote an edge and its image in the geometric realization, we define $\alpha : [0, length(e)] \rightarrow |e|$ to be the arclength parametrization, so that $d_G(u,v)=|\alpha^{-1}(v)-\alpha^{-1}(u)|$ for any $u,v \in |e|$. This definition may then be extended to any two points in $|G|$ by restricting a given path from one point to another to edges in $G$, adding up these lengths, then taking the distance to be the minimum length of any such path.
In this way, all points along an edge are points in a metric graph, not just the original graph's vertices.

A \emph{system of loops of $G$} refers to a set of cycles whose associated homology classes form a minimal generating set for the $1$-dimensional (singular) homology group of $G$.
The \emph{length-sequence} of a system of loops is the sequence of lengths of elements in this set listed in non-decreasing order. Thus, a system of loops of $G$ is \emph{shortest} if its length-sequence is lexicographically smallest among all possible systems of loops of $G$.

One particular class of metric graphs we will be working with are \emph{bouquet graphs}.
These are metric
graphs containing a single vertex with a number of self-loops of 
various lengths attached to it.%

\subsection{Intrinsic \Cech{} and persistence distortion distances}
\label{sec:distancesbground}

In this section, we recall the distances between metric graphs that are being explored in this work. We note that both are actually \emph{pseudo-distances} because it can be the case that $d(G_1, G_2)=0$ when $G_1 \neq G_2$. However, for ease of exposition, we will refer to them simply as distances in this paper.
Both rely on the \emph{bottleneck distance} on the space of persistence diagrams, a version of which we now state.

\begin{defn} Let $X$ and $Y$ be persistence diagrams with $\mu: X \rightarrow Y$ a bijection. The \textbf{bottleneck distance} between $X$ and $Y$ is
\[d_B(X,Y) := \displaystyle{\inf_{\mu:X \rightarrow Y} \supx_{x\in X} ||x - \mu(x)||_1}.\]
\end{defn}
\noindent Although this definition differs from the standard version of the bottleneck distance, which uses $||x - \mu(x)||_{\infty}$ rather than $||x - \mu(x)||_1$, the two are related via the inequalities $||x||_{\infty} \leq ||x||_1 \leq 2 ||x||_{\infty}$. 

Next, let $(G, d_G)$ be a metric graph with geometric realization $|G|$. Define the intrinsic ball $B(x, a_i) = \{ y \in |G| : d_G(x,y) \leq a_i\}$ for any $x \in |G|$, as well as the uncountable open cover $U_{a_i} = \{ B(x, a_i) : x \in |G| \}$. We use $\C{a_i}$ to denote the nerve of the cover $U_{a_i}$, referred to as the \emph{intrinsic \Cech{} complex}.
See Figure~\ref{fig:ex_cech} for an illustration.
Then $\{ \C{a_i} \hookrightarrow \C{a_j} \}_{0 \leq a_i < a_j}$ is the \emph{intrinsic \Cech{} filtration} inducing the \emph{intrinsic \Cech{} persistence module} $\{H_k(\C{a_i}) \rightarrow H_k(\C{a_j}) \}_{0 \leq a_i < a_j}$ in any dimension $k,$ and the corresponding persistence diagram is denoted $Dg_kIC_G$.
The following intrinsic \Cech{} distance definition comes from \cite{ChazalSilvaOudot2014}. Here, we work with dimension $k = 1$.
\begin{figure}[ht]
\begin{center}
\includegraphics[scale=0.5]{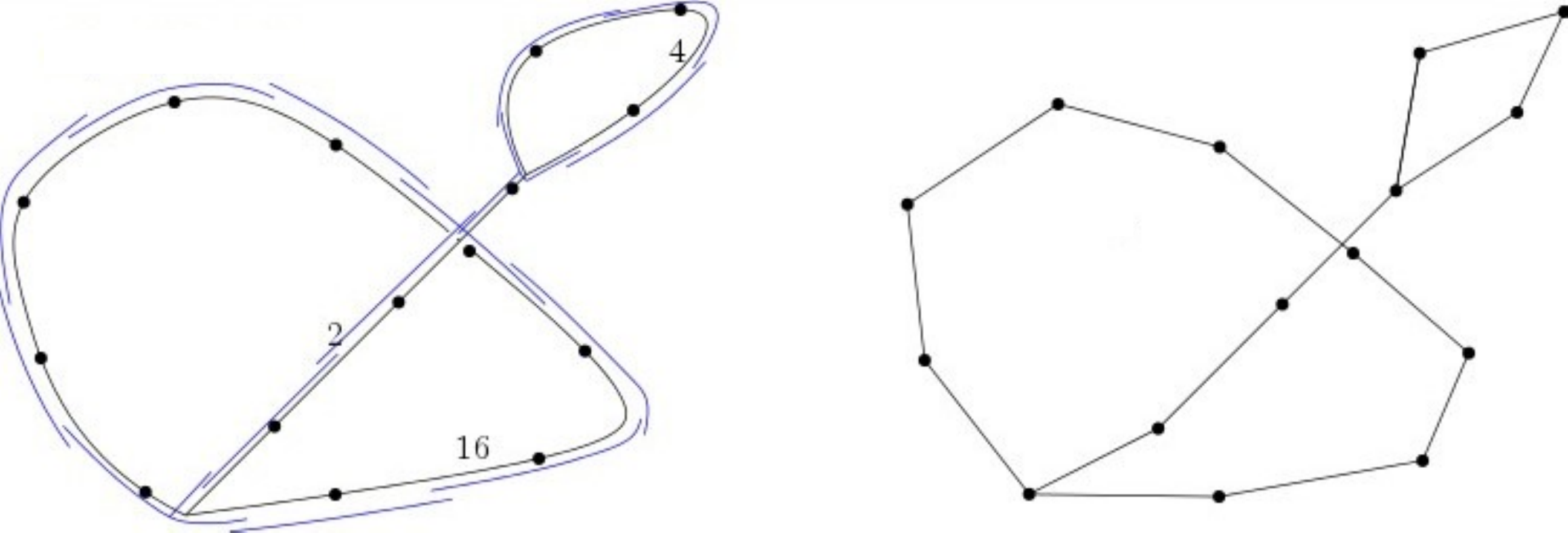}
\caption{A finite subset of the infinite cover at a fixed radius (left) and its  corresponding nerve (right).}\label{fig:ex_cech}
\end{center}
\end{figure}
\begin{defn}
Given two metric graphs $(G_1, d_{G_1})$ and $(G_2, d_{G_2})$, their \textbf{intrinsic \v{C}ech distance} is
$ d_{IC}(G_1, G_2) := d_B(\mathrm{Dg}_1 IC_{G_1}, \mathrm{Dg}_1 IC_{G_2}).$
\end{defn}

The persistence distortion distance was first introduced in \cite{DeyShiWang2015}. Given a base point $v \in |G|$, define the geodesic distance function $f_v : |G| \rightarrow \R$ where $f_v(x) = d_G(v,x)$. Then $\Dg{f_v}$ is the union of the $0-$ and $1-$dimensional \emph{extended persistence diagrams} for $f_v$ (see \cite{cohen2009extending} for the details of extended persistence). Equivalently, it is the $0$-dimensional \emph{levelset zigzag persistence diagram} induced by $f_v$ \cite{CarlssonSilvaMorozov2009}. Define $\Phi : |G| \rightarrow SpDg$, $\Phi(v) = \DgPD{f_v},$ where $SpDg$ denotes the space of persistence diagrams for all points $v \in |G|$.
The set $\Phi(|G|) \subset SpDg$ is the \emph{persistence distortion} of the metric graph $G$.

\begin{defn}
Given two metric graphs $(G_1, d_{G_1})$ and $(G_2, d_{G_2})$, their \textbf{persistence distortion distance} is
\[ d_{PD}(G_1, G_2) := d_H(\Phi(|G_1|), \Phi(|G_2|)) \]
where $d_H$ denotes the Hausdorff distance. In other words,
\[
d_{PD}(G_1, G_2) = \max \left\{ \sup_{D_1 \in \Phi(|G_1|)} \infd_{D_2 \in \Phi(|G_2|)} d_B(D_1, D_2),  \sup_{D_2 \in \Phi(|G_2|)} \infd_{D_1 \in \Phi(|G_1|)} d_B(D_1, D_2)\right\}.
\]
\end{defn}
\noindent Note that the diagram $\DgPD{f_v}$ contains both $0-$  and $1-$dimensional persistence points, but only points of the same dimension are matched under the bottleneck distance. 
In this paper, we will only focus on the points in the
$1$-dimensional extended persistence diagrams for the persistence distortion  distance computation.

\section{Calculating the intrinsic \Cech{} distance}
\label{sec:icdist}

In this section, we show that the intrinsic \Cech{} distance between two metric graphs may be easily computed from knowing the shortest systems of loops for the graphs.
We begin with a theorem that characterizes the bottleneck distance between two sets of points in the extended plane.

\begin{thm}
\label{thm:ICDist}
Let $D_1= \{ (0, a_1), \ldots, (0, a_n) \}$ and $D_2= \{ (0, b_1), \ldots, (0, b_n) \}$ be two persistence diagrams with $0 \leq a_1 \leq \cdots \leq a_n$ and $0 \leq b_1 \leq \cdots \leq b_n$, respectively.
Then $d_B(D_1,D_2)= \ds\max_{i=1}^n |a_i - b_i|$.
\end{thm}

\begin{proof}
To simplify notation, we use the convention that for all $i=1,\ldots,n$, $(0,a_i) = \ol{a_i}$, $(0, b_i) = \ol{b_i}$, and $(0,0) = \ol{0}$.
Let $\mu$ be any matching of points in $D_1$ and $D_2$, 
where each point $\ol{a_i}$ in $D_1$ is either matched to a unique point $\ol{b_j}$ in $D_2$ or to the nearest neighbor in the diagonal (and similarly for $D_2$).
Assume that $C_\mu$ is the cost of the matching $\mu$, i.e., the maximum distance between two matched points.

Now, let $\mu^*$ be the matching such that $\mu^*(\ol{a_i}) = \ol{b_i}$ for all $0 \leq i \leq n$.
By construction, the cost of this matching is $C_{\mu^*} = \displaystyle\max_{i=1}^{n}|a_i-b_i|$.
We claim that the matching cost of $\mu^*$ is less than or equal to that of $\mu$, i.e., $C_{\mu^*} \leq C_\mu$.
If this is the case, then $\mu^*$ is the optimal bottleneck matching and therefore $d_B(D_1,D_2) = C_{\mu^*}$.

To show this, we look at where the matchings $\mu$ and $\mu^*$  differ.
Note that since all of the off-diagonal points in $D_1$ and $D_2$ lie on the $y$-axis, any such point matched to the diagonal under $\mu$ may simply be matched to $(0,0)$ since this will yield the same value in the $\ell_1-$norm.
Now, starting with $b_1$, let $j$ be the first index where $\mu(\ol{a_j}) \neq \ol{b_j}$.  Then, we have two cases:  (1) $\mu(\ol{a_k}) = \ol{b_j}$ for some $k > j$ (i.e., $\ol{b_j}$ is matched with some $\ol{a_k} \neq \ol{a_j}$); or (2) $\mu(\ol{0}) = \ol{b_j}$ (i.e., $\ol{b_j}$ is matched with the diagonal, or equivalently, to $\ol{0}$).  We show that in either case, matching $\ol{b_j}$ with $\ol{a_j}$ instead does not increase the cost of the matching.

In the first case, let us also assume that $\mu(\ol{a_j}) = \ol{b_l}$ for some $l > j$  (the situation where $\mu(\ol{a_j})=\ol{0}$ will be taken care of in the second case).  Then, $\max\{|a_j-b_j|,|a_k-b_l|\} \leq \max\{|a_j-b_l|,|a_k-b_j|\}$.  That is, if we were to instead pair $\ol{a_j}$ with $\ol{b_j}$ and $\ol{a_k}$ with $\ol{b_l}$, the cost of the matching would be lower. This can be seen by working through a case analysis on the relative order of $a_j,a_k,b_j$, and $b_l$ along the $y$-axis. Intuitively, we can think of $a_j,a_k,b_j$, and $b_l$ as the four corners of a trapezoid as in Figure \ref{fig:trap}.  The diagonals of the trapezoid represent the distances under the matching $\mu$, while the legs of the trapezoid represent the distances when we pair $\ol{a_j}$ with $\ol{b_j}$ and $\ol{a_k}$ with $\ol{b_l}$.  The maximum of the lengths of the legs will always be less than the maximum of the lengths of the diagonals.  Adjusting the lengths of the top and bottom bases (which amounts to changing the order of $a_j,a_k,b_j$, and $b_l$ along the $y$-axis) does not change this fact.  Therefore, matching $\ol{b_j}$ with $\ol{a_j}$ instead of $\ol{a_k}$ does not increase the cost of the matching.

\begin{figure}[h]
\centering
\includegraphics[width=0.3\textwidth]{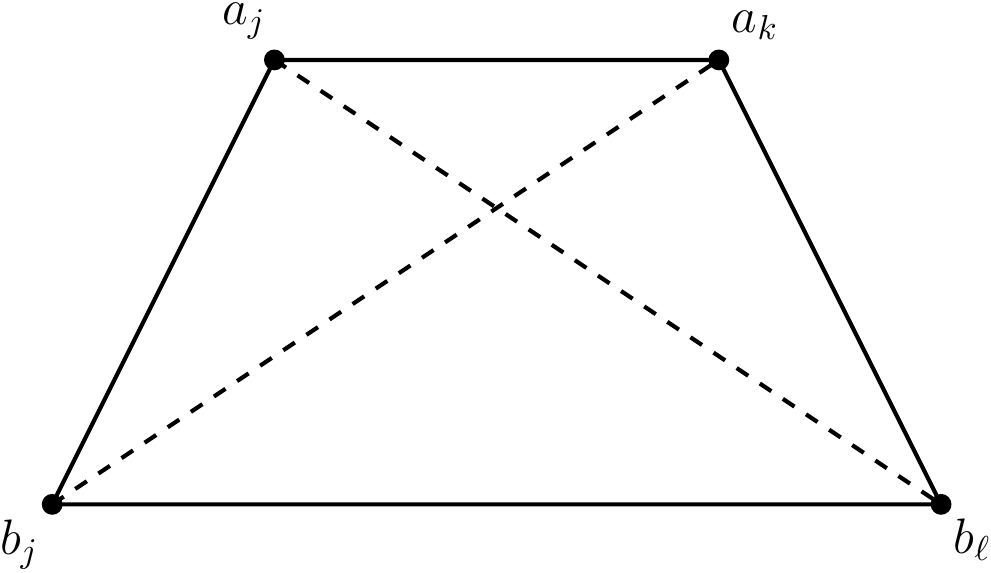}
\caption{A trapezoid formed by $a_j,a_k,b_j$, and $b_l$.}
\label{fig:trap}
\end{figure}

In the second case, if $\ol{b_j}$ is matched to $\ol{0}$, there must be some $\ol{a_k}$ with $k \geq j$ that is matched to $\ol{0}$, as well. If we were to instead match $\ol{b_j}$ to $\ol{a_k}$, this does not increase the cost of the matching since $\max\{ b_j, a_k\} \ge |a_k - b_j|$ (i.e., the original cost is greater than the new cost). After this rematching, $\ol{b_j}$ is no longer matched to $\ol{0}$ and this reverts to the first case. Similarly, if $\ol{a_j}$ is matched to $\ol{0}$, it may be rematched in a similar manner.

By looking at all the pairings where $\mu$ and $\mu^*$ differ (in increasing order of indices), pairing $\ol{a_i}$ with $\ol{b_i}$ instead of $\mu(\ol{a_i})$ (and similarly, pairing $\ol{b_i}$ with $\ol{a_i}$ rather than what it was paired with under $\mu$) always results in the same or lower cost matching.
Therefore, $C_{\mu^*} 
\leq 
C_{\mu}$ for all matchings $\mu$; hence, $d_B(D_1,D_2) = C_{\mu^*} = \displaystyle\max_{i=1}^{n}|a_i-b_i|$.
\end{proof}

To see how this applies to the computation of the intrinsic \Cech{} distance between two metric graphs, let $G_1$ be a metric graph with a shortest system of $m$ loops of lengths $0 < 2t'_1\leq\cdots\leq 2t'_m$, and let $G_2$ be a metric graph with a shortest system of $n$ loops of lengths $0 < 2s_1\leq \cdots \leq 2s_n$. Without loss of generality, suppose $n \geq m$.
From \cite{GasparovicGommelPurvine2018}, the $1$-dimensional intrinsic \Cech{} persistence diagrams of $G_1$ and $G_2$ are
the multisets of points $\mathrm{Dg}_1 IC_{G_1}=
\left\{ \ds \left(0, \frac{t'_1}{2}\right), \ldots, \left(0, \frac{t'_m}{2}\right) \right\}$ and $\mathrm{Dg}_1 IC_{G_2}=
\left\{ \ds \left(0, \frac{s_1}{2}\right), \ldots, \left(0, \frac{s_n}{2}\right) \right\}$.
In order to apply Theorem~\ref{thm:ICDist}, we add $n-m$ copies of the point $(0,0)$ at the start of the list of points in $\mathrm{Dg}_1 IC_{G_1}$, i.e., let
\[\mathrm{Dg}_1 IC_{G_1}=\left\{ \ds\left(0, \frac{{t}_1}{2}\right), \ldots, \left(0, \frac{{t}_n}{2}\right) \right\},\] where $t_1 = \cdots = t_{n-m} = 0$, $t_{n-m+1} =t'_1,\ldots,$ and $t_n=t'_m$.

\begin{cor}\label{cor:ICDist}
Let $G_1$ and $G_2$ be as above. Then
 \[d_{IC}(G_1, G_2) = \ds\max_{i=1}^n \frac{|s_i - t_i|}{2}.\]
\end{cor}

\section{Relating the intrinsic \Cech{} and persistence distortion distances for a bouquet graph and an arbitrary graph}
\label{sec:main}

\subsection{Feasible regions in persistence diagrams}
\label{subsec:feasible}

Our eventual goal for our main theorem (Theorem~\ref{thm:main}) is to estimate a lower bound for the persistence distortion distance between metric graphs $G_1=(V_1,E_1)$ and $G_2=(V_2,E_2)$ so that we can compare it with the intrinsic \Cech{} distance between them, given in Corollary~\ref{cor:ICDist}. A fundamental part of this process relies on the notion of a \emph{feasible region} for a point in a given persistence diagram lying on the $y$-axis.

\begin{defn}
The \textbf{feasible region} for a point $\ol{s} := (0,s) \in \R^2$ is defined as
\[ \feas{s} = \{ z=(z_1,z_2) : 0 \leq z_1 \leq z_2, s \leq z_2 \leq z_1+s\}. \]
\end{defn}
\noindent An illustration of a feasible region is shown in Figure \ref{fig:feas}.

\begin{figure}[ht]
\centering
\includegraphics[width=0.7\textwidth]{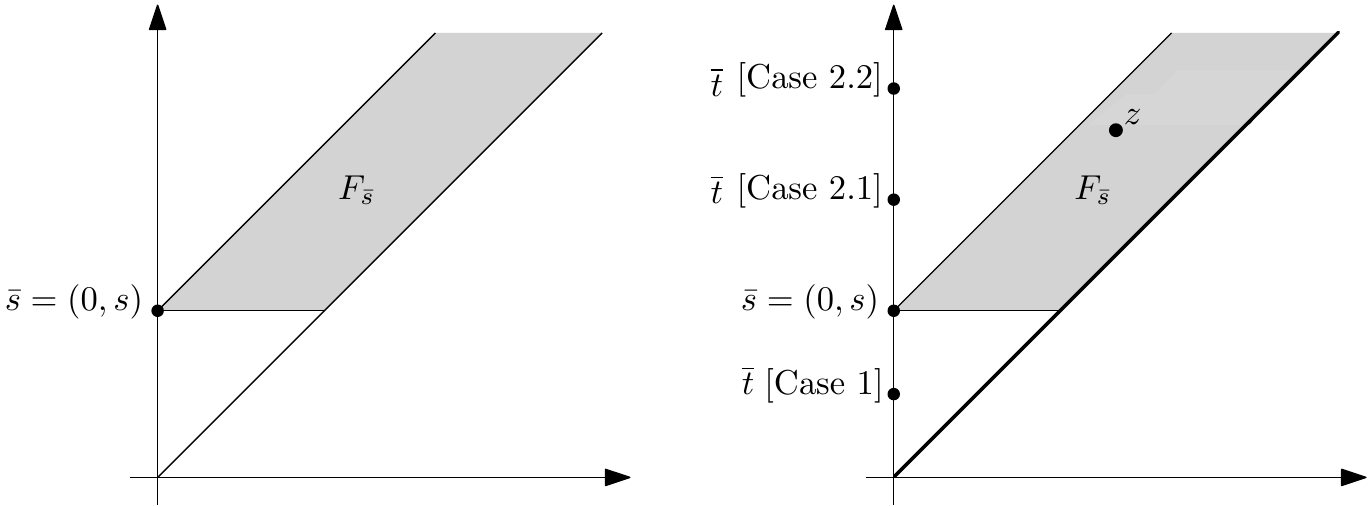}
\caption{Left: An illustration of the feasible region for $\ol{s}$. Right: The three cases within the proof of Lemma \ref{lem:DistPts}.}
\label{fig:feas}
\end{figure}
The following lemma establishes an important property of feasible regions that will be used later in the proof of the main theorem.
\begin{lem}\label{lem:DistPts}
Given any point $z \in F_{\ol{s}}$ and any point $\ol{t}=(0,t)$, $||\ol{s}-\ol{t}||_1 \leq ||z-\ol{t}||_1$.
\end{lem}
\begin{proof}
We proceed with a simple case analysis using the definition of $F_{\ol{s}}$.
Let $z = (z_1, z_2)$.
\begin{description}\denselist
\item[Case 1:] Assume $s \geq t$ so that $||\ol{s}-\ol{t}||_1=s-t$. By the definition of $\feas{s}$, we have $z_2 \geq s$ and thus
\[||z-\ol{t}||_1= z_1+z_2-t \geq z_1+s-t \geq s-t=||\ol{s}-\ol{t}||_1.\]
\item[Case 2.1:] If $s < t$, then $||\ol{s}-\ol{t}||_1 =t-s$. If $t \leq z_2$, then since $z_1 \geq z_2-s$ and $z_2 \geq s$,
\[ ||z-\ol{t}||_1 = z_1+z_2-t \geq (z_2-s)+z_2-t \geq t-s+t-t = t-s = ||\ol{s}- \ol{t}||_1.\]
\item[Case 2.2:] If $s < t$ but $t > z_2$, then since $z_2\leq z_1+s$, it follows that
\[||z-\ol{t}||_1= z_1+t-z_2 \geq z_1+t-(z_1+s) = t-s=||\ol{s}-\ol{t}||_1.\]
\end{description}
The lemma now follows.
\end{proof}

\subsection{Properties of the geodesic distance function for an arbitrary metric graph}
\label{subsec:arbitrary}

Let $G = (V, E)$ be an arbitrary metric graph with shortest system of loops of lengths $2s_1, \cdots, 2s_n$.
Fix an arbitrary base point $v \in |G|$ and consider $\Dg{f_v}$, as defined in Section~\ref{sec:distancesbground}.
Let $T_v$ denote the shortest path tree in $G$ rooted at $v$.
We consider the base point $v \in |G|$ to be a graph node of $G$; that is, we add it to $V$ if necessary. We further assume that the graph $G$ is ``generic'' in the sense that there do not exist two or more shortest paths from the base point $v$ to any graph node of $G$ in $V$.
For any input metric graph $G$, we can perturb it to be one that is generic within arbitrarily small Gromov-Hausdorff distance.

For simplicity, when $v$ is fixed, we shall omit $v$ in our notation and speak of the persistence diagram $D :=\Dg{f_v}$, the function $f := f_v$, and the shortest path tree $T := T_v.$

We present three straightforward observations, the first of which follows immediately from the definition of the shortest path tree and the Extreme Value Theorem.

\begin{obs}
\label{lem:DHatNo}
The shortest path tree $T$ of $G$ has $|V|-1$ edges, and there are $|E|-|V|+1$ non-tree edges.
For each non-tree edge $e \in E \setminus T$, there exists a unique $u \in e$ such that  $f(u)$ is a local maximum value of $f$.
\end{obs}
Figure~\ref{fig:shortest-path-tree} contains an example of a metric graph on the left and the corresponding shortest path tree illustrating Observation~\ref{lem:DHatNo} in the middle.

\begin{figure}[ht]
\centering
\includegraphics[width=0.7\linewidth]{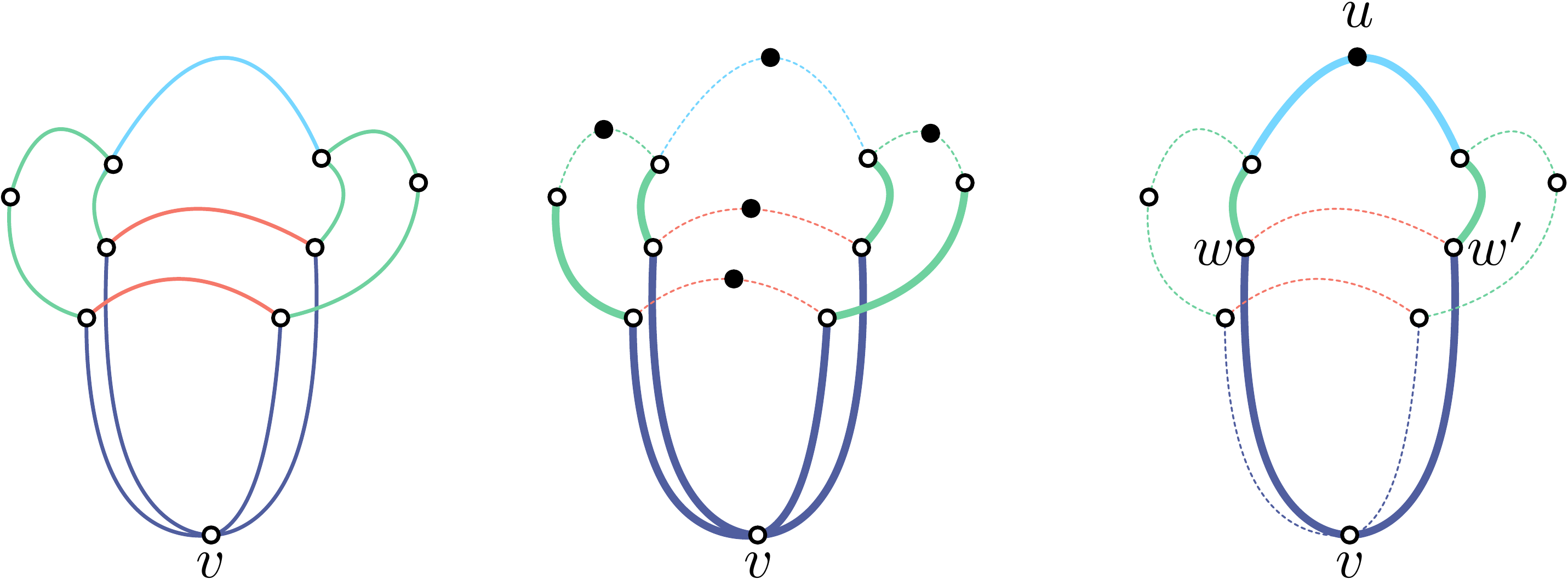}
\caption{Left: An example of a metric graph $G$. The dark blue, light blue, red, and green edges are of lengths $4$, $3$, $2$, and $1$, respectively, and $v$ is the base point.
Middle: An illustration of Observation~\ref{lem:DHatNo} for the shortest path tree $T$ of $G$.
Thick blue and green edges are part of the shortest path tree $T$, and the others are the non-tree edges.
Black points represent local maxima of $f$.
Right: An illustration of Observation~\ref{lem:DHatMax1}. Thick edges are part of $\gamma$, $f(v) = p$, and $p'$ is either $f(w)$ or $f(w')$ depending on the relative positions of $w$ and $w'$.}
\label{fig:shortest-path-tree}
\end{figure}

Note that every feature in the persistence diagram $D$ must be born at a point in the graph that is an up-fork, i.e., a point coupled with a pair of adjacent directions along which the function $f$ is increasing.
Since there are no local minimum points of $f$ (except for $v$ itself), these must be vertices in the graph of degree at least 3 (see, e.g., \cite{OudotSolomon2018}).

\begin{obs}\label{lem:DHatGen1}
The diagram $D$ is a multiset of points $\{(p_i,f(u_i)) \, | \, i \in \{1,\ldots, |E| - |V| + 1\}\}$, where for every $i$, $p_i=f(w)$ for some graph node $w \in V$.
\end{obs}

The final observation relates to points belonging to cycles in $G$ that yield local maximum values of $f$ (see~\cite{AgarwalEdelsbrunnerHarer2006}).

\begin{obs}
\label{lem:DHatMax1}
Let $\gamma$ be an arbitrary cycle in $G$. If $u$ is the point in $\gamma$ corresponding to the largest local maximum value of $f$, let $p$ be the lowest function value of $f$ for all points in $\gamma$.
Then there is a point in the persistence diagram $D$ of the form $(p',f(u))$, where $p' \geq p.$
\end{obs}
\noindent See Figure~\ref{fig:shortest-path-tree} (right) for an illustration.

To delve further into this, let $\{\gamma_1,\ldots,\gamma_n\}$ denote the elements of the shortest system of loops for $G$ listed in order of non-decreasing loop length.

\begin{lem}
\label{lem:gvtake2}
Consider a cycle $\gamma=\gamma_{i_1} + \ldots + \gamma_{i_m}$ ($m \leq n$), where each $\gamma_{i_k}$ ($1 \leq k \leq m$) is an element of the shortest system of loops for $G$ and $i_1\leq i_2 \leq \ldots \leq i_m$.
Suppose the edge $e \in \gamma$ contains the point $u$ in $\gamma$ with the largest local maximum value of $f$.
Then $f(u)\geq s_{i_m}$, where $s_{i_m}$ is half the length of cycle $\gamma_{i_m}$ in the shortest system of loops.
\end{lem}

\begin{proof}
Since each $\gamma_{i_k}$ ($1 \leq k \leq m$) is an element of the shortest system of loops for $G$ and $i_1\leq i_2 \leq \ldots \leq i_m$, this implies that $s_{i_1} \leq \cdots \leq s_{i_m}$, where $2s_{i_k}$ is the length of cycle $\gamma_{i_k}$ in the shortest system of loops of $G$.

Assume instead that $f(u)<s_{i_m}$.
Now, $\gamma$ in $G$ must contain at least one non-tree edge as it is a cycle.
Let $e_1, \ldots e_\ell = e$ be all non-tree edges of $G$ with largest function value at most $f(u)$. Assume they contain maximum points $u_1,\ldots,u_\ell=u$, respectively, where the edges and maxima are sorted in order of increasing function value of $f$.

For two points $x, y \in |T|$, let $\alpha(x,y)$ denote the unique tree path from $x$ to $y$ within the shortest path tree.
For each $j\in\{1,\ldots,\ell\}$, let $e_j=(e^0_j,e^1_j)$ and let $c_j$ denote the cycle $c_j=\alpha(v,e^1_j)\circ e_j \circ \alpha(e^0_j,v)$.
By assumption, since $u =u_\ell$ is the point in $\gamma$ with the largest local maximum value of $f$ and $f(u)<s_{i_m}$, it follows that the length of every cycle $c_j$ is less than $s_{i_m}$.
However, the set of cycles $\{c_1,\ldots,c_\ell\}$ form a basis for the subgraph of $G$ spanned by all edges containing only points of function value at most $f(u)$.
Therefore, we may represent $\gamma$ as a linear combination of cycles from the set $\{c_1, \ldots, c_\ell \}$, i.e., $\gamma$ may be decomposed into shorter cycles, each of length less than $s_{i_m}  =\ds\frac{length(\gamma_{i_m})}{2}$.
This is a contradiction to the fact that $\gamma_{i_1},\ldots,\gamma_{i_m}$ are elements of the shortest system of loops for $G$.
Hence, we conclude that $f(u)\geq s_{i_m}$.
\end{proof}

\begin{figure}[ht]
\centering
\includegraphics[width=0.25\linewidth]{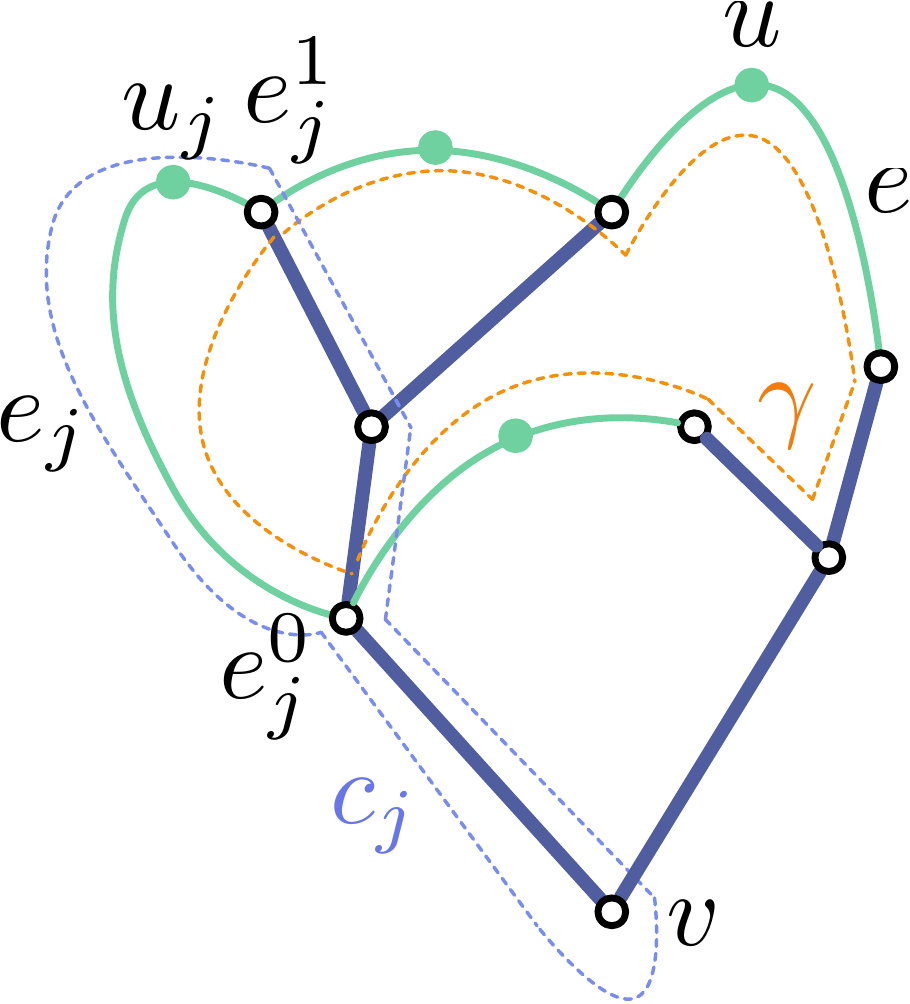}
\caption{An illustration of the proof of Lemma~\ref{lem:gvtake2}. In this case, $\gamma$ is the sum of three smaller cycles and there are four non-tree edges highlighted in green. One $c_j$ is shown corresponding to the local maximum $u_j$.
}
\label{fig:lemma6}
\end{figure}

An example that illustrates the proof of Lemma~\ref{lem:gvtake2} is shown in Figure~\ref{fig:lemma6}. Later we will use the following simpler version of Lemma \ref{lem:gvtake2}, where $\gamma$ is a single element of the shortest system of loops.

\begin{cor}\label{lem:DHatMaxVsSi}
Let $\gamma$ be an element of the shortest system of loops for $G$ with a length $2s$, and let $u$ denote the point in any edge of $\gamma$ with largest maximum value of $f$. Then $f(u) \geq s$.
\end{cor}

\subsection{The main theorem and its proof}

We are now ready to establish a comparison of the intrinsic \Cech{} and persistence distortion distances between a bouquet metric graph and an arbitrary metric graph.

\begin{thm}
\label{thm:main}
Let $G_1$ and $G_2$ be finite metric graphs such that $G_1$ is a bouquet graph
and $G_2$ is arbitrary.
Then
\[ d_{IC}(G_1, G_2) \leq \ds \frac{1}{2}d_{PD}(G_1, G_2). \]
\end{thm}

\begin{proof}
Let $G_1$ be a bouquet graph consisting of $m$ cycles of lengths $0 < 2t'_1\leq\ldots\leq 2t'_m$, all sharing one common point $o \in |G_1|$.
Let $G_2$ be an arbitrary metric graph with shortest system of loops consisting of $n$ loops of lengths $2s_1,\cdots,2s_n$ listed in non-decreasing order.
In what follows, we suppose $n \geq m$; the case when $m \geq n$ proceeds similarly. As before, we obtain a sequence of length $n$,
$2t_1 \leq 2t_2 \cdots \leq 2t_n$ (where $t_1 = \cdots = t_{n-m} = 0$, $t_{n-m+1} =t'_1,\cdots,$ and $t_n=t'_m$).
Let $f$ and $g$ denote the geodesic distance functions on $G_1$ and $G_2$, respectively.

First, as in Corollary~\ref{cor:ICDist}, the intrinsic \Cech{} distance between $G_1$ and $G_2$, denoted by $\delta$, is
 \begin{equation}
 \label{eq:delta}\delta:=d_{IC}(G_1, G_2) = \ds\max_{i=1}^n \frac{|s_i - t_i|}{2}.\end{equation}

Second, note that the persistence diagram $D_1:=\Dg{f_o}$ with respect to the base point $o$ is $D_1=\{(0, t_1), \cdots, (0, t_n)\}$ (of course, this may include some copies of $(0,0)$ if $m < n$).
Next, fix an arbitrary base point $v \in |G_2|$ and consider the persistence diagram $D_2 := \Dg{g_v}$.
Consider the abstract persistence diagram $D^\star:= \{(0,s_1), \cdots, (0,s_n) \}=\{\ol{s_1},\ldots,\ol{s_n}\}$ that consists only of points on the $y$-axis at the $s_i$ values.
Unless $G_2$ is also a bouquet graph, $D^\star$ is not necessarily in $\Phi(|G_2|)$.
Nevertheless, we will use this persistence diagram as a point of comparison and relate points in $D_2$ to $D^\star$.
Notice that a consequence of Theorem~\ref{thm:ICDist} is that
\begin{equation}
\label{claim1}
 d_B(D_1, D^\star) = \ds\max_{i=1}^n |s_i - t_i| = 2\delta.
\end{equation}

In order to accomplish our objective of relating points in $D_2$ with points in the ideal diagram $D^\star$, we need the following lemma relating to feasible regions, which were introduced in Section~\ref{subsec:feasible}.

\begin{lem}
\label{lem:feasPts}
Let $D' = \{z_1, \ldots, z_n\}$ be an arbitrary persistence diagram such that $z_i \in \feas{s_i}$.
Then $d_B(D_1, D^\star) \leq d_B(D_1, D')$.
\end{lem}
\begin{proof}
Consider the optimal bottleneck matching between $D_1$ and $D'$.
According to Lemma~\ref{lem:DistPts}, if the point $\ol{t_j}=(0,t_j) \in D_1$ is matched to $z_i \in D'$ under this optimal matching, the matching of $\ol{s_i}=(0,s_i)\in D^\star$ to $\ol{t_j}$ will yield a smaller distance.
In other words, the induced bottleneck matching between $D_1$ and $D^\star$, which is equal to $2\delta$, can only be smaller than $d_B(D_1,D')$.
\end{proof}

The outline of the remainder of the proof of Theorem \ref{thm:main} is as follows. Theorem \ref{thm:perfect} shows that one can assign points in $D_2$ to the points in $D^\star$ in such a way that the condition in Lemma~\ref{lem:feasPts} is satisfied.
The fact that one can assign points in the fixed persistence diagram $D_2$ to the distinct feasible regions $\feas{s_i}$ relies on the series of structural observations and results in Section~\ref{subsec:arbitrary}, along with an application of Hall's marriage theorem.
Finally, the inequality in Lemma~\ref{lem:feasPts} and the definition of the persistence distortion distance imply that
\begin{equation}
2\delta=d_B(D_1, D^\star) \leq \inf_{v \in |G_2|} d_B(D_1,D_2) \leq d_{PD}(G_1,G_2),
\end{equation}
which, together with \eqref{eq:delta}, completes the proof of Theorem~\ref{thm:main}.

The following theorem establishes the existence of a one-to-one correspondence between points in $D^\star$ and points in $D_2$. The goal is to construct a bipartite graph $\widehat{G} = (D^\star, D_2, \widehat{E})$, where there is an edge $\hat{e} \in \widehat{E}$ from $\ol{s_i}\in D^\star$ to $z \in D_2$ if and only if $z \in F_{\ol{s_i}}$. To prove the theorem, we 
invoke Hall's marriage theorem, which requires showing that for any subset $S$ of points in $D^\star$, the number of neighbors of $S$ in $D_2$ is at least $|S|$.

\begin{thm}\label{thm:perfect}
The graph $\widehat{G}$ contains a perfect matching.
\end{thm}
\begin{proof}
For simplicity, let $T = T_v$ and $g = g_v$.
First, note that there is a one-to-one correspondence $\Psi: E_2 \setminus T \to D_2$ between the set of non-tree edges in $G_2$ (each of which contains a unique maximum point of $g$) and the set of points in $D_2$. In particular, from Observations \ref{lem:DHatNo} and \ref{lem:DHatGen1}, the death-time of each point in $ D_2$ uniquely corresponds to a local maximum $u_e$ within a non-tree edge $e$ of $G_2$.

Fix an arbitrary subset $S \subseteq D^\star$ with $|S|=a$.
In order to apply Hall's marriage theorem, we must show that there are at least $a$ neighbors of $S$ in $\widehat{G}$. We achieve this via an iterative procedure which we now describe. The procedure begins at step $k=0$ and will end after $a$ iterations. Elements in $S = \{\ol{s_{i_1}}, \ldots, \ol{s_{i_a}}\}$ are processed in non-decreasing order of their values, which also means that $i_1 < i_2 < \cdots < i_a$. At the start of the $k$-th iteration, we will have processed the first $k$ elements of $S$, denoted $S_k=\{\ol{s_{i_1}},\ldots,\ol{s_{i_k}}\}$, where for each $\ol{s} := \ol{s_{i_h}}\in S_k$ that we have processed ($1 \leq h \leq k$), we have maintained the following three invariances:

\begin{description}\denselist
\item[\textbf{Invariance 1:}] $\bar{s}$ is associated to a unique edge $e_{\bar{s}}\in E_2\setminus T$ containing a unique maximum
$u_{e_{\bar{s}}}$ such that $\Psi(e_{\bar{s}}) \in D_2$ is a neighbor of $\bar{s}$. We say that $e_{\bar{s}}$ and $u_{e_{\bar{s}}}$ are \emph{marked by $\bar{s}$}.
\item[\textbf{Invariance 2:}]: $\bar{s}$ is also associated to a cycle $\widetilde{\gamma}_{h} = \gamma_{i_h} + \sum \gamma_\ell$ (where the sum ranges over all $\ell$ belonging to some index set $J_h \subset \{1,\ldots, i_{h}-1\}$), such that $e_{\ol{s}}$ contains the point in $\widetilde{\gamma}_{h}$ with the largest value of $g$.
\item[\textbf{Invariance 3:}]: $\myheight(\widetilde{\gamma}_{h}) \le s_{i_h} $, where $\myheight(\gamma) =\displaystyle \max_{x\in\gamma} g(x) - \min_{x\in \gamma} g(x)$ represents the height (i.e., the maximal difference in the $g$ function values) of a given loop $\gamma$.
\end{description}

\noindent Set $\overline{S_k} = S \setminus S_k = \{\overline{s_{i_{k+1}}},\ldots,\overline{s_{i_{a}}}\}$, denoting the remaining elements from $S$ to be processed.
Our goal is to identify a new neighbor in $D_2$ for element $\overline{s_{i_{k+1}}}$ from $\overline{S_k}$ satisfying the three invariances. Once we have done so, we will then set $S_{k+1} = S_k \cup \{s_{i_{k+1}}\}$ and move on to the next iteration in the procedure.

Note that $\overline{s_{i_{k+1}}}$ corresponds to an element $\gamma_{i_{k+1}}$ of the shortest system of loops for $G_2$.
Let $e$ be the edge in $\gamma_{i_{k+1}}$ containing the maximum $u_e$ of highest $g$ function value among all edges in $\gamma_{i_{k+1}}$.
There are now two possible cases to consider, and we will demonstrate how to obtain a new neighbor for $\overline{s_{i_{k+1}}}$ in either case.

In the first case, suppose $u_e$ is not yet marked by a previous element in $S$. In this case, $e_{\overline{s_{i_{k+1}}}}=e$ and $\widetilde{\gamma}_{i_{k+1}} = \gamma_{i_{k+1}}$.
We claim that the point $(p_e, g(u_e))$ in the persistence diagram $D_2$ corresponding to the maximum $u_e$ is contained in the feasible region $F_{\overline{s_{i_{k+1}}}}$.
In other words,  $s_{i_{k+1}}\leq g(u_e) \leq  p_e+s_{i_{k+1}}$.
Indeed, by Lemma~\ref{lem:gvtake2}, $s_{i_{k+1}}\leq g(u_e)$, and by Observation~\ref{lem:DHatMax1},
\[g(u_e)-s_{i_{k+1}}\leq lowest(\gamma_{i_{k+1}})\leq p_e,\]
where $lowest(\gamma_{i_{k+1}}):= \ds\min_{x\in \gamma_{i_{k+1}}} g(x)$.
Thus, $(p_e,g(u_e))\in D_2$ is a new neighbor for $\ol{s_{i_{k+1}}}\in S$ since it is contained in $F_{\ol{s_{i_{k+1}}}}$. Consequently, we mark $e$ and $u_e$ by $\ol{s_{i_{k+1}}}$ and continue with the next iteration.

In the second case, the maximum point $u_e$ has already been marked by a previous element $s_{j_1} \in S_k$ and been associated to a cycle $\widetilde{\gamma}_{j_1}$. Observe that $s_{j_1} \le s_{i_{k+1}}$ since our procedure processes elements of $S$ in non-decreasing order of their values (and thus $j_1 < i_{k+1}$).
We must now identify an edge other than $e$ for $s_{i_{k+1}}$ satisfying the three invariance properties.
To this end, let $\widehat{\gamma}_1 = \gamma_{i_{k+1}} + \widetilde{\gamma}_{j_1}$, and let $e_1$ be the edge containing the maximum in $\widehat{\gamma}_1$ with largest function value. If $e_1$ is unmarked, we set $e_{\overline{s_{i_{k+1}}}}=e_1$. Otherwise, if $e_1$ is marked by some cycle $\gamma_{j_2}$, we construct the loop $\widehat{\gamma}_2=\widehat{\gamma}_1+\widetilde{\gamma}_{j_2}=\gamma_{i_{k+1}}+\widetilde{\gamma}_{j_1}+\widetilde{\gamma}_{j_2}$.
We continue this process until we find $\widehat{\gamma}_\eta=\gamma_{i_{k+1}}+\widetilde{\gamma}_{j_1}+\widetilde{\gamma}_{j_2}+ \ldots + \widetilde{\gamma}_{j_\eta}$ such that the edge $e_\eta$ containing the point of maximum function value of $\widehat{\gamma}_\eta$ is not marked.
Once we arrive at this point, we set $\widetilde{\gamma}_{i_{k+1}} = \widehat{\gamma}_\eta$ and $e_{\overline{s_{i_{k+1}}}}=e_\eta$, so that the edge $e_\eta$ and corresponding maximum $u_{e_\eta}$ are marked by $\overline{s_{i_{k+1}}}$.

The reason that the procedure outlined above must indeed terminate is as follows. Each time a new $\widetilde{\gamma}_{j_\nu}$ is added to a cycle $\widehat{\gamma}_{j_{\nu-1}}$ (for $\nu\in \{1,\ldots,\eta\}$), it is because the edge containing the maximum point of $\widehat{\gamma}_{j_{\nu-1}}$ with largest function value is marked by $\overline{s_{j_\nu}}$.
Note that $j_\nu\neq j_\beta$ for $\nu\neq \beta$ (as during the procedure, the edge $e_i$ containing the maximum function value in the cycle $\widehat{\gamma}_i$ are all distinct), each $j_\nu < i_{k+1}$, and $\overline{s_{j_\nu}} \in S_k$. Furthermore, Invariance 2 guarantees that $\widehat{\gamma}_\eta$ cannot be empty, as each cycle $\widetilde{\gamma}_{j_\nu}$ can be written as a linear combination of elements in the shortest system of loops with indices at most $j_\nu$. As ${j_\nu} < i_{k+1}$, the cycle $\gamma' = \widetilde{\gamma}_{j_1}+\widetilde{\gamma}_{j_2}+ \ldots + \widetilde{\gamma}_{j_\eta}$ can be represented as a linear combination of basis cycles with indices strictly smaller than $i_{k+1}$. In other words, $\gamma_{i_{k+1}}$ and $\gamma'$ must be linearly independent, and thus $\widehat{\gamma}_\eta = \gamma_{i_{k+1}} + \gamma'$ cannot be empty.
Again, $j_\nu\neq j_\beta$ for $\nu\neq \beta$ and each $j_\nu < i_{k+1}$, and thus it follows that after at most $k$ iterations, we will obtain a cycle whose highest valued maximum and corresponding edge are not yet marked.

Now, we must show that the three invariances are satisfied as a result of the process described in this second case. To begin, we point out that Invariance 2 holds by construction. Next, the following lemma establishes Invariance 3.

\begin{lem}
\label{lem:gvineq}
For $\widetilde{\gamma}_{i_{k+1}} = \widehat{\gamma}_\eta = \gamma_{i_{k+1}}+\widetilde{\gamma}_{j_1}+\widetilde{\gamma}_{j_2}+ \ldots + \widetilde{\gamma}_{j_\eta}$ as above, $\myheight(\widetilde{\gamma}_{i_{k+1}}) \le s_{i_{k+1}}$.
\end{lem}

\begin{proof}
Set $\widehat{\gamma}_0 = \gamma_{i_{k+1}}$, and for $\nu \in \{1,\ldots,\eta\}$, set $\widehat{\gamma}_\nu = \gamma_{i_{k+1}} + \widetilde{\gamma}_{j_1} + \cdots + \widetilde{\gamma}_{j_\nu}$.
Using induction, we will show that $\myheight(\widehat{\gamma}_\nu) \le s_{i_{k+1}}$ for any $\nu\in \{0,\ldots, \eta\}$.
The inequality obviously holds for $\nu=0$.
Suppose it holds for all $\nu \le \rho < \eta$,
and consider $\nu = \rho+1$ where $\widehat{\gamma}_{\rho+1} = \widehat{\gamma}_\rho + \widetilde{\gamma}_{j_{\rho+1}}$.
The cycle $\widetilde{\gamma}_{j_{\rho+1}}$ is added as the edge $e_\rho$ of $\widehat{\gamma}_{\rho}$ containing the current maximum point of highest value of $g$ has already been marked by $\overline{s_{j_{\rho+1}}}$ with $j_{\rho+1} < i_{k+1}$. By Invariance 2, $e_\rho$ must also be the edge in $\widetilde{\gamma}_{j_{\rho+1}}$ containing the point of maximum $g$ function value, which we denote by $g(e_\rho)$.
Therefore, after the addition of $\widehat{\gamma}_{\rho}$ and $\widetilde{\gamma}_{j_{\rho+1}}$,
\begin{align}
\text{(i)}~ &highest(\widehat{\gamma}_{\rho+1}) := \max_{x\in \widehat{\gamma}_{\rho+1}} g(x) \le g_v(e_\rho), ~\text{and} \label{eqn:maxbound}\\
\text{(ii)}~&lowest(\widehat{\gamma}_{\rho+1}) := \min_{x\in \widehat{\gamma}_{\rho+1}} g(x) \ge \min \{~ lowest(\widehat{\gamma}_\rho), ~lowest(\widetilde{\gamma}_{j_{\rho+1}}) ~\}.\nonumber
\end{align}
By the induction hypothesis, $\myheight(\widehat{\gamma}_{\rho}) \le s_{i_{k+1}}$, while by Invariance 3, $\myheight(\widetilde{\gamma}_{j_{\rho+1}}) \le s_{j_{\rho+1}} \le s_{i_{k+1}}$.
By (ii) of equation (\ref{eqn:maxbound}), it then follows that
\[lowest(\widehat{\gamma}_{\rho+1}) \ge \min \{ g(e_\rho) - \myheight(\widehat{\gamma}_\rho), g(e_\rho) - \myheight(\widetilde{\gamma}_{j_{\rho+1}})\} \ge g(e_\rho) - s_{i_{k+1}}. \]
Combining this with (i) of equation (\ref{eqn:maxbound}), we have that $\myheight(\widehat{\gamma}_{\rho+1}) \le s_{i_{k+1}}$. The lemma then follows by induction.
\end{proof}

Finally, we show that Invariance 1 also holds.  Since $\widetilde{\gamma}_{i_{k+1}} = \widehat{\gamma}_\eta = \gamma_{i_{k+1}}+\gamma'$, with $\gamma'$ defined as above, by Lemma \ref{lem:gvtake2}, we have that $g(u_{e_\eta}) \ge s_{i_{k+1}}$.
 Suppose $u_{e_\eta}$ is paired with some graph node $w$
 so that $p_{e_\eta}=g(w)$. As the height of $\widetilde{\gamma}_{i_{k+1}}$ is at most $s_{i_{k+1}}$ (Lemma \ref{lem:gvineq}), combined with Observation~\ref{lem:DHatMax1}, we have that
\[ g(u_{e_\eta})-s_{i_{k+1}} \leq lowest(\widetilde{\gamma}_{i_{k+1}}) \leq p_{e_\eta}.\]
This implies that the point $(p_{e_\eta}, g(u_{e_\eta})) \in F_{\overline{s_{i_{k+1}}}}$, establishing Invariance 1.

We continue the process described above until $k = a$. At each iteration, when we process $\overline{s_{i_k}}$, we add a new neighbor for elements in $S$. In the end, after processing all of the $a$ elements in $S$, we find $a$ neighbors for $S$, and the total number of neighbors in $\widehat{G}$ of elements in $S$ can only be larger. Since this holds for any subset $S$ of $D^\star$, the condition for Hall's theorem is satisfied for the bipartite graph $\widehat{G}$. This implies that there exists a perfect matching in $\widehat{G}$, completing the proof of Theorem~\ref{thm:perfect}.
\end{proof}

Theorem~\ref{thm:main} now follows from Lemma~\ref{lem:feasPts} and equation~\eqref{eq:delta}.
\end{proof}


\subsection{Proving the conjecture when both graphs are trees of loops}

The techniques of Theorem~\ref{thm:main} are specific to the case when one of the graphs is a bouquet graph. However, we can prove Conjecture~\ref{mainCon} in another setting, as well: the case when both graphs are \emph{trees of loops}.

\begin{defn}
A \textbf{tree of loops} is a metric graph constructed via wedge sums of cycles and edges.
 \end{defn}

\noindent See Figure \ref{fig:TreeOfLoops} for an example. The fact that the inequality holds when both graphs are trees of loops follows from an application of the following lemma. 

\begin{figure}[ht]
\centering
\includegraphics[width=0.3\textwidth]{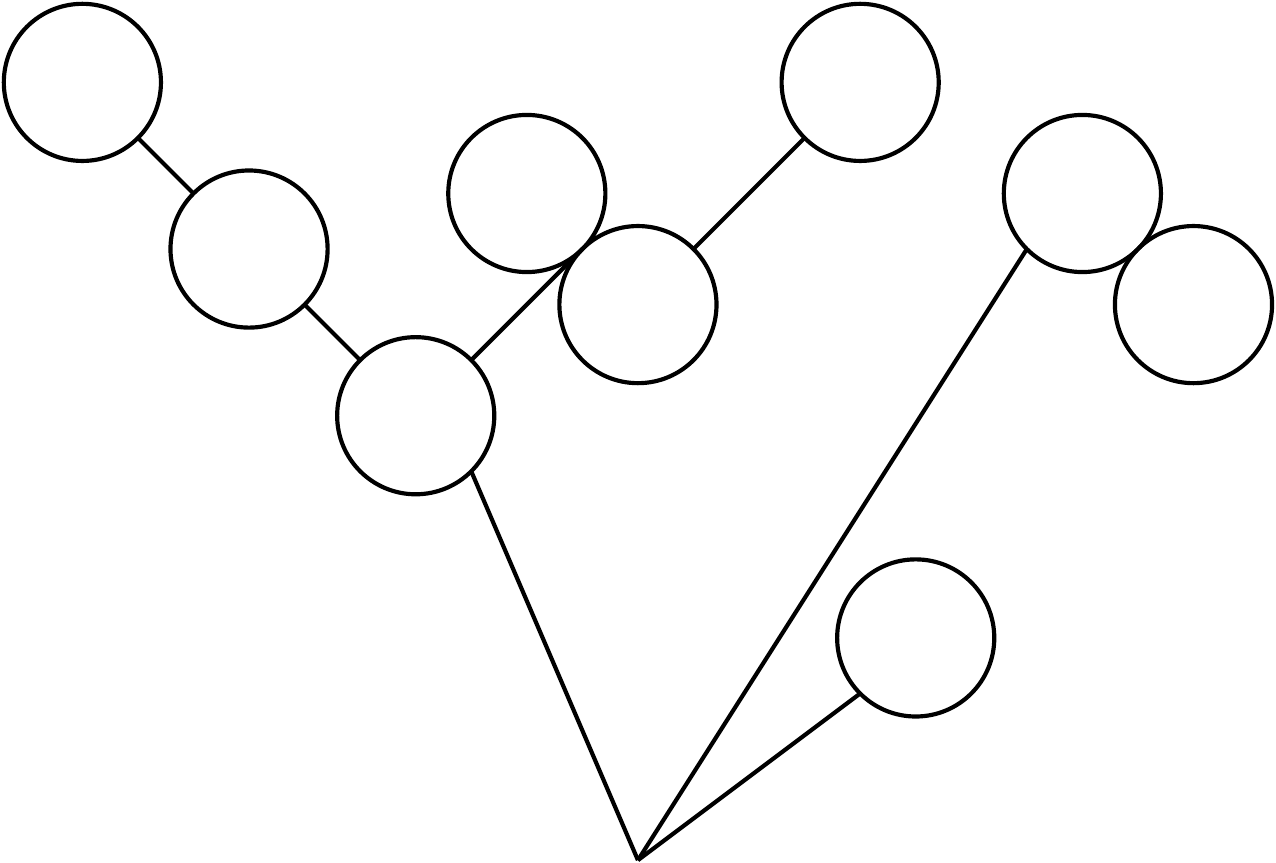}
\caption{An example of a tree of loops.}
\label{fig:TreeOfLoops}
\end{figure}

\begin{lem}\label{lem:bottleneck_ineq}
Let $P$ and $Q$ be two persistence diagrams with finite numbers of off-diagonal points. Let $d_1$ and $d_2$ be distances defined between points in $P$ and points in $Q$ such that $d_1(p, q) \leq d_2(p, q)$ for every $p \in P, q \in Q$. Then $d_B(P, Q)$ under distance $d_1$ is less than or equal to $d_B(P, Q)$ under distance $d_2$.
\end{lem}

\begin{proof}
The bottleneck distance under a particular distance $d$ is given by
\[ d_B(P, Q) = \min_\mu \max_{p} d(p, \mu(p)),\]
where the minimum is taken over all matchings $\mu : P \rightarrow Q$. If we consider a fixed matching $\mu$, the relationship between $d_1$ and $d_2$ implies that
\begin{equation}
\max_{p} d_1(p, \mu(p)) \leq \max_{p} d_2(p, \mu(p)).\label{this} \end{equation}
Let $\mu'$ be the matching that achieves the minimum for distance $d_2$.
The inequality \eqref{this} together with this minimum implies that
\[ d_B(P, Q)\text{ under }d_1 = \min_\mu \max_{p} d_1(p, \mu(p)) \leq \max_{p} d_1(p, \mu'(p)) \leq \max_{p} d_2(p, \mu'(p))  = d_B(P,Q)\text{ under }d_2. \]
\end{proof}

 \begin{prop}
 \label{prop:treeofloops}
Let $G_1$ and $G_2$ be two finite metric graphs such that both are trees of loops. Then
\[ d_{IC}(G_1, G_2) \leq \ds \frac{1}{2} d_{PD}(G_1, G_2). \]
\end{prop}

\begin{proof}
Let $G_1$ and $G_2$ be trees of loops of lengths $2t_1',\ldots,2t'_m$ and $2s_1,\ldots,2s_n$, respectively, each set listed in non-decreasing order. Without loss of generality, suppose $n \geq m$.
First, as in Corollary~\ref{cor:ICDist}, the intrinsic \Cech{} distance between $G_1$ and $G_2$ is $
 d_{IC}(G_1, G_2) = \ds\max_{i=1}^n \frac{|s_i - t_i|}{2}$
 where $t_1 = \cdots = t_{n-m} = 0$, $t_{n-m+1}=t'_1,\ldots,$ and $t_n=t'_m$.
Suppose $d_{IC}(G_1,G_2)=\ds\frac{|t_k-s_k|}{2}$ for some $k$, $1\leq k \leq n$.
Let $f$ and $g$ denote the geodesic distance functions on $G_1$ and $G_2$, respectively.

For trees of loops, the persistence diagrams take the form $\DgPD{f_v}=\{(p_i,p_i+t_i)\}_{1\leq i \leq n}$ and $\DgPD{g_w}=\{(q_i,q_i+s_i)\}_{1\leq i \leq n}$ for $v \in G_1$ and $w \in G_2$. The proposition holds if, for any pair of persistence diagrams $\DgPD{f_{v}}$ and $\DgPD{g_{w}}$, $d_B(\DgPD{f_v},\DgPD{g_w})\geq |t_k-s_k|$.

We will prove this by applying Lemma \ref{lem:bottleneck_ineq}.
For any $i,j\in\{1,\ldots,n\}$, let $d_1((p_i, p_i+t_i), (q_j, q_j+s_j)) = |t_i-s_j|$ and $d_2((p_i, p_i+t_i), (q_j, q_j+s_j)) = || (p_i, p_i+t_i) - (q_j, q_j+s_j)||_1.$ It's easy to show that $d_1 \leq d_2$ for every pair of points, so that the conditions of the lemma are satisfied. Notice that distance $d_1$ is equivalent to the case where all $p_i =q_i = 0$, i.e., the points are along the $y$-axis. By Theorem~\ref{thm:ICDist}, the bottleneck distance under $d_1$ equals $|t_k - s_k|=2d_{IC}(G_1, G_2)$.
Therefore, the bottleneck distance between $\DgPD{f_v}$ and $\DgPD{g_w}$ under $d_2$ is at least $|t_k-s_k|$, as desired.
\end{proof}

\section{Discussion and future work}
\label{sec:future}

In this paper, we compare the discriminative capabilities of the intrinsic \Cech{} and persistence distortion distances, which are based on topological signatures of metric graphs. The intrinsic \Cech{} signature arises from the intrinsic \Cech{} filtration of a metric graph, and the persistence distortion signature is based on the set of persistence diagrams arising from sublevel set filtrations of geodesic distance functions from all base points in a given metric graph.
A map from a metric graph to these topological signatures is not injective: 
two different metric graphs may map to the same signature.
However, each signature captures structural information of a graph and serves as a type of topological summary.
Understanding the relationship between the intrinsic \Cech{} and persistence distortion distances enables one to better understand the discriminative powers of such summaries.

We conjecture that the intrinsic \Cech{} distance is less discriminative than the persistence distortion distance for general metric graphs $G_1$ and $G_2$, so that there exists a constant $c\geq 1$ with $d_{IC}(G_1,G_2) \le c \cdot d_{PD}(G_1,G_2)$.
This statement is trivially true in the case when both graphs are trees as the intrinsic \Cech{} distance is 0 while the persistence distortion distance is not.
We establish a sharper version of the conjectured inequality in the case when one of the graphs is a bouquet graph and the other is arbitrary, as well as in the case when both graphs are obtained via wedges of cycles and edges.
The methods of proof in Theorem~\ref{thm:main} and Proposition~\ref{prop:treeofloops} rely on explicitly knowing the forms of the persistence diagrams for the geodesic distance function in the case of a bouquet graph or a tree of loops.
Therefore, these methods do not readily carry over to the most general setting for arbitrary metric graphs.
Nevertheless, we believe that the relationship between the intrinsic \Cech{} and persistence distortion distances should hold for arbitrary finite metric graphs.
Intuitively, the intrinsic \Cech{} signature only captures the sizes of the shortest loops in a metric graph, whereas the persistence distortion signature takes into consideration the relative positions of such loops and their interactions with one another.

As one example application relating the intrinsic \Cech{} and persistence distortion summaries (and hence, distances), the work of Pirashvili, et al. \cite{PirashviliSteinbergBelchi-Guillamon2018} considers how the topological structure of chemical compounds relates to solubility in water, which is of fundamental importance in modern drug discovery. Analysis with the topological tool mapper \cite{SinghMemoliCarlsson2007} reveals that compounds with a smaller number of cycles are more soluble. The number of cycles, as well as cycle lengths, is naturally encoded in the intrinsic \Cech{} summary. 
In addition, these authors also use a discrete persistence distortion summary -- where only the graph nodes, i.e., the atoms, serve as base points -- to show that nearby compounds have similar levels of solubility. Although we conjecture that the intrinsic \Cech{} distance is less discriminative then the persistence distortion distance, it might be sufficient in this particular analysis since solubility is highly correlated with the number of cycles of a chemical compound, that is, with the intrinsic \Cech{} summary \cite{GasparovicGommelPurvine2018}.
It would be interesting to investigate other applications of the intrinsic \Cech{} and persistence distortion summaries in the context of data sets modeled by metric graphs.

In addition, recall from the definition of the persistence distortion distance the map $\Phi : |G| \rightarrow SpDg$, $\Phi(v) = \DgPD{f_v}$.
The map $\Phi$ is interesting in its own right.
For instance, what can be said about the set $\Phi(|G|)$ in the space of persistence diagrams for a given $G$?
Given only the set $\Phi(|G|) \subset SpDg$, what information can one recover about the graph $G$?
Oudot and Solomon~\cite{OudotSolomon2018} show that there is a dense subset of metric graphs (in the Gromov--Hausdorff topology, and indeed an open dense set in the so-called fibered topology) on which their barcode transform via the map $\Phi$ is globally injective up to isometry.
They also prove its local injectivity on the space of metric graphs.
Another question of interest is, how does the map $\Phi$ induce a stratification in the space of persistence diagrams?
Finally, it would also be worthwhile to compare the discriminative capacities of the persistence distortion and intrinsic \Cech{} distances to other graph distances, such as the interleaving and functional distortion distances in the special case of Reeb graphs.

\section*{Conflict of Interest Statement}

On behalf of all authors, the corresponding author states that there is no conflict of interest.

\bibliographystyle{plainurl}
\bibliography{main}

\end{document}